\documentclass[12pt]{amsart}

\usepackage{amscd,amsthm,amsfonts,amssymb,esint}
\usepackage{amsmath}
\usepackage[colorlinks=true, urlcolor=black, bookmarks=true, bookmarksopen=true, citecolor=black, linkcolor=black]{hyperref}
\usepackage{fullpage}
\usepackage{geometry}
    \newcommand{\BA}{{\mathbb {A}}} 
    \newcommand{\BC}{{\mathbb {C}}} 
     \newcommand{\BF}{{\mathbb {F}}}

    \newcommand{\BO}{{\mathbb {O}}} 
    \newcommand{\BQ}{{\mathbb {Q}}} \newcommand{\BR}{{\mathbb {R}}}

     \newcommand{\BZ}{{\mathbb {Z}}}

    \newcommand{\CA}{{\mathcal {A}}}

     \newcommand{\CH}{{\mathcal {H}}}

    \newcommand{\CO}{{\mathcal {O}}} 
     
    \newcommand{\CS}{{\mathcal {S}}}

    \newcommand{\fa}{{\mathfrak{a}}}

     \newcommand{\fP}{{\mathfrak{P}}}

     \newcommand{\fX}{{\mathfrak{X}}}
     \newcommand{\fZ}{{\mathfrak{Z}}}

    \newcommand{\End}{{\mathrm{End}}} \newcommand{\Eis}{{\mathrm{Eis}}}

    \newcommand{\Ker}{{\mathrm{Ker}}}

    \newcommand{\PGL}{{\mathrm{PGL}}}

    \renewcommand{\Re}{{\mathrm{Re}}}

    \newcommand{\Rat}{{\mathrm{Rat}}}

    \newcommand{\tr}{{\mathrm{tr}}}

    \newcommand{\wt}{\widetilde}
    \newcommand{\wh}{\widehat}

    \newcommand{\lra}{\longrightarrow}
    \newcommand{\ra}{\rightarrow} 
    \newcommand{\lto}{\longmapsto}\newcommand{\bs}{\backslash}

    \theoremstyle{plain}
    \newtheorem{thm}{Theorem}[section] \newtheorem{cor}[thm]{Corollary}
      \newtheorem{prop}[thm]{Proposition}

 \newtheorem{claim}[thm]{Claim}

\theoremstyle{remark} \newtheorem{remark}[thm]{Remark}
\theoremstyle{remark} 
\theoremstyle{remark} 

    \numberwithin{equation}{section}

    \newcommand{\cusp}{\mathrm{cusp}}

     \newcommand{\un}{{\mathrm{un}}}

\begin{document}

\title{Isolation of the Cuspidal Spectrum: the Function Field Case}

\author{Li Cai}
\address{Yau Mathematical Sciences Center\\
Tsinghua University\\
Beijing, 100084, People's Republic of China}
\email{lcai@mail.tsinghua.edu.cn}

\author{Bin Xu}
\address{School of Mathematics\\
Sichuan University\\
No. 29 Wangjiang Road\\ 
Chengdu, 610064, People's Republic of China}
\email{binxu@scu.edu.cn}


\begin{abstract} 
Isolating cuspidal automorphic representations from the whole automorphic spectrum is a basic problem in the trace formula approach.  
For example, matrix coefficients of supercupidal representations can be used as test functions for this, which kills the continuous spectrum, but also a large class of cuspidal automorphic representations. 
For the case of number fields, multipliers of the Schwartz algebra is used in the recent work \cite{Iso} to isolate all cuspidal spectrum which provide enough test functions and suitable for the comparison of orbital integrals.  
These multipliers are then applied to the proof of the Gan-Gross-Prasad conjecture for unitary groups \cite{Iso, BCZ20}. 
In this article, we prove similar result on isolating the cuspidal spectrum in \cite{Iso} for the function field case.
\end{abstract}

\subjclass[2010]{Primary 11F70; Secondary 11F72.}
\keywords{}
\thanks{L. Cai is partially supported by NSFC grant No.11971254. B. Xu is partially supported by NSFC grant No.11501382}

\maketitle

\section{Isolation of the cuspidal spectrum}\label{sec: intro}

The spectral expansion in (relative) trace formulae is a central but difficult problem. 
Simple trace formulae can largely simplify this problem. 
The traditional simple (relative) trace formulae use matrix coefficients
of supercuspidal representations as (local components of) test functions, which exclude many important cases. 
In \cite{LV}, Lindenstrauss and Venkatesh introduce a new type of simple trace
formula to prove the Weyl's law for spherical cusp forms on locally symmetric spaces associated to a split adjoint semisimple group $G$ over $\BQ$. 
Their approach is based on the observation that there are strong relations between the spectrum of the Eisenstein series at different places.

Recently, in \cite{Iso}, the authors develop a new technique for isolating components on the spectral side of the trace formula.
Precisely, they introduce an analogue of the Bernstein center at Archimedean places, and construct enough multipliers preserving matching of test functions by considering Schwartz functions, instead of smooth functions with compact supports. 
Using these multipliers, one can isolate the cuspidal spectrum without the full spectral decomposition, and establish the refined Gan-Gross-Prasad conjecture for a large class of representations (see also \cite{BCZ20}), which is also important to the work \cite{LTXZZ} on the Beilinson-Bloch-Kato conjecture for certain Rankin-Selberg motives.

The goal of this article is to give a proof of the result on isolating the cuspidal spectrum (see \cite[Theorem 1.1]{Iso}) for the function field case. 
Similar to the number field case, the result here is expected to be applied to general situation of trace formula approach over function fields. 
This work came to be through an effort to understand the paper \cite{Iso}, and will be the starting point of our project on the trace formula approach for the arithmetic problems over function fields.

Let $F$ be the function field of a smooth projective and geometrically connected curve over the finite field $\BF_q$.  
Denote by $\BA = \BA_F$ the adele ring of $F$.
Let $G$ be a connected reductive group over $F$ and $Z$ the center of $G$. 
Let $S_G$ be the set of all the primes of $F$ such that $G(F_v)$ is ramified. 
We fix a maximal compact subgroup $K_0$ of $G(\BA)$, and a Haar measure 
$\mathrm{d}g = \prod_v \mathrm{d}g_v$ on $G(\BA)$, such that $K_{0,v}$ is hyperspecial maximal with volume one under $\mathrm{d}g_v$ for every place $v$ not in $S_G$.

Take a unitary automorphic character $\omega: Z(F) \bs Z(\BA) \ra \BC^\times$. 
We define $L^2(G(F) \bs G(\BA))_\omega$
to be the $L^2$-completion of the space of smooth functions $\varphi$ on $G(\BA)$ satisfying
\begin{itemize}
	\item $\varphi(z\gamma g) = \omega(z) \varphi(g)$ for every $z \in Z(\BA)$, $\gamma \in G(F)$
		and $g \in G(\BA)$;
	\item $|\varphi|^2$ is integrable on $Z(\BA) G(F) \bs G(\BA)$.
\end{itemize}
Denote by $L^2_0(G(F) \bs G(\BA))_\omega$ the subspace of $L^2(G(F) \bs G(\BA))_\omega$ consisting of
functions $\varphi$ which are cuspidal, i.e. the constant term 
$$\varphi_P(g)=\int_{U(F)\bs U(\BA)}\varphi(ug)\,\mathrm{d}u$$
is zero for all proper parabolic subgroups $P$ of $G$, here $U$ is the unipotent radical of $P$.
The group $G(\BA)$ acts on $L^2(G(F) \bs G(\BA))_\omega$ via the right regular action $R$, 
and $L^2_0(G(F) \bs G(\BA))_\omega$ is closed under this action.
Denote by $C_c^\infty(G(\BA))$ the algebra (without an identity element) 
of smooth functions on $G(\BA)$ with compact supports. Then
the action $R$ induces an action of $C_c^\infty(G(\BA))$ on $L^2(G(F) \bs G(\BA))_\omega$ by
\[R(f) = \int_{G(\BA)} f(g)R(g)\,\mathrm{d}g, \quad (f \in C_c^\infty(G(\BA))).\]

Let $S$ be a set of places of $F$ containing $S_G$. 
Let $K \subset K_0$ be an open compact group of the form 
$K = \prod_{v \in S} K_v \times \prod_{v \not \in S} K_{0,v}=K_S\times K_0^{S}$.
Assume the character $\omega$ is invariant under the action of $K \cap Z(\BA)$.
Denote by $L^2(G(F) \bs G(\BA) /K)_\omega$ the subspace of $L^2(G(F) \bs G(\BA))_\omega$ 
invariant under the action of $K$ via $R$. 
Similarly, we have the space $L^2_0(G(F) \bs G(\BA)/K)_\omega$ consisting of cuspidal functions. 
Denote by $C_c^\infty(K \bs G(\BA) /K)$ the algebra of bi-$K$-invariant functions
in $C_c^\infty(G(\BA))$. Then $C_c^\infty(K \bs G(\BA) /K)$ acts on both $L^2(G(F) \bs G(\BA)/K)_\omega$
and $L^2_0(G(F) \bs G(\BA)/K)_\omega$ via $R$.

For every place $v \not\in S_G$, let $\CH_v = C_c^\infty(K_{0,v} \bs G(F_v)/K_{0,v})$ 
be the spherical Hecke algebra of $G_v$ with respect to $K_{0,v}$. 
Let $\CH^{(S)}$ be the restricted tensor product of $\CH_{v}$ for places $v \not\in S$.  
Then $\CH^{(S)}$ can be viewed as a subalgebra of $C_c^\infty(K \bs G(\BA) /K)$ by the embedding $f^{(S)} \mapsto 1_{K_S} \otimes f^{(S)}$, where $f^{(S)} \in \CH^{(S)}$, and $1_{K_S}$ is the characteristic function of $K_S$. 
In particular, the Hecke algebra $\CH^{(S)}$ acts on both $L^2(G(F) \bs G(\BA)/K)_\omega$ and 
$L^2_0(G(F) \bs G(\BA)/K)_\omega$ via $R$.

Let $\pi = \otimes_v \pi_v$ be an irreducible admissible representation of $G(\BA)$. 
Then the algebra $C_c^\infty(G(\BA))$ also acts on $\pi$ by
\[\pi(f) = \int_{G(\BA)} f(g)\pi(g)\,\mathrm{d}g, \quad (f \in C_c^\infty(G(\BA))).\]
Denote by $\pi^K$ the invariant subspace of $\pi$ under $K$. 
Then $C_c^\infty(K \bs G(\BA) /K)$ acts on $\pi^K$.

Assume that $\pi^K$ is non-zero. 
In particular, $\pi_v$ is unramified for all $v\notin S$. 
We call such a representation $\pi$ is {\em $(G,S)$-CAP} if there is a proper parabolic subgroup $P$ of $G$ and a cuspidal automorphic representation $\sigma$ of $M(\BA)$, where $M$ is the Levi
part of $P$, such that for all but finitely many places $v$ of $F$ not in $S$, the unramified 
representation $\pi_v$ is a constituent of $I_P^G(\sigma_v)$.  
Here, $I_P^G(\sigma_v)$ denotes the normalized parabolic induction of $\sigma_v$.

\begin{thm}\label{iso-cusp} 
	Suppose that $\pi$ is an irreducible admissible representation of $G(\BA)$ 
	which is not $(G,S)$-CAP. Then there exists a Hecke algebra element $\mu \in \CH^{(S)}$ such that 
	\begin{itemize}
		\item $R(\mu)$ maps $L^2(G(F) \bs G(\BA) /K)_\omega$ into $L^2_0(G(F) \bs G(\BA)/K)_\pi$.
		      Here, $L^2_0(G(F) \bs G(\BA)/K)_\pi$ is the $\pi$-nearly isotypic subspace of
		      $L^2_0(G(F) \bs G(\BA)/K)_\omega$, i.e. the direct sum of irreducible $G(\BA)$-sub-representations $\pi'$ such that $\pi'_v\simeq \pi_v$ for almost all places $v$ of $F$.
		\item $\pi(\mu)$ is the identity map on $\pi^K$.
	\end{itemize}
\end{thm}

Comparing with the number field case, the proof of the above result is much more simpler. 
In the number field case, there are two main ingredients:
\begin{enumerate}
\item (Killing the continuous spectrum)
Let $M$ be the Levi part of a proper parabolic subgroup $P$ of $G$, and $\sigma$ be a cupsidal automorphic representation on $M(\BA)$. 
To kill the continuous spectrum coming from $\sigma$, one needs to work with Schwartz functions, and construct a (global) multiplier on the algebra of Schwartz functions to annihilate all the representations $I_P^G(\sigma_\lambda)$, where $\lambda$ runs over (a subspace of) the complex space $\fa_M^* = \Rat(M) \otimes_\BZ \BC$ (see the notation in Section \ref{sec:spec-decomp}).
Let $v$ be a non-Archimedean place such that $\sigma_v$ is unramified. For each spherical Hecke function $f \in \CH_v$, the function
\[\fa_M^*\ni \lambda \mapsto \tr (I_P^G(\sigma_{\lambda,v})(f))\]
is of $q_v$-exponential type by the Satake isomorphism. 
To construct the desired global multiplier, one has to show that there are enough multipliers at Archimedean places to match with the above functions of exponential type from non-Archiemdean places. 
This is the hardest part (see \cite[Theorem 2.11]{Iso}). 
\item (Isolating cuspidal representations) By a theorem of Donnelly (\cite{Don}), the Casimir eigenvalues of all cuspidal automorphic representations of $G(\BA)$ with given central character distributes discretely, so that one can isolate the spectrum using certain Weierstrass product from complex analysis (see \cite[\S 2.1]{Iso}). 
\end{enumerate}
While in the function field case, one has:
\begin{itemize}
\item[(1')] (Killing the continuous spectrum)
One can just work with Hecke algebras in the function field case.
For a global function field $F$, the cardinalities of residue fields of all its local fields are powers of a common number, i.e. the cardinality of the constant field. 
In particular, by the Satake isomorphism, for two places $v_1$ and $v_2$ with
	$q_{v_2}$ being a power of $q_{v_1}$,  the function
	\[\fa_M^*\ni\lambda \mapsto \tr (I_P^G(\sigma_{\lambda,v_2})(f)) \quad (f \in \CH_{v_2})\]
	should come from some spherical functions at $v_1$ (see 
	Claim \ref{claim T_{w,w'}} in Section \ref{sec: proof} for example).
	 Using this observation, we can construct the desired global Hecke algebra element which annihilates the continuous spectrum coming from $\sigma$. 
\item[(2')] (Isolating cuspidal representations) By a theorem of Harder (\cite{Harder}), all the cuspidal automorphic representations of $G(\BA)$ with given central character and level are finite, then one can isolate the spectrum using polynomials of Hecke algebra elements. 
\end{itemize}

We will present the case of $G=\PGL_2$ in Section \ref{sec: PGL2} to give an overview of the proof, and go to the general case in Section \ref{sec:spec-decomp} -- \ref{sec: proof}.

\subsection*{Acknowledgement} We thank Professor Ye Tian for his consistent encouragement, and thank Professor Yifeng Liu for reading the preliminary version of our manuscript. We also thank the anonymous referees for both the careful reading of our manuscript, and the very useful comments and suggestions. L. Cai is partially supported by NSFC grant No.11971254, and B. Xu is partially supported by NSFC grant No.11501382.

\section{Example: The case of $\PGL_2$}\label{sec: PGL2}

In this section, we prove Theorem \ref{iso-cusp} for $G = \PGL_2$ as an example. 
The main ingredients for the proof are the following:
\begin{itemize}
	\item {\it Spectral decomposition along the cuspidal supports:}
		we have the spectral decomposition of unitary $G(\BA)$-modules
		\[L^2(G(F) \bs G(\BA)) = \left( \bigoplus_\pi L_\pi^2 \right) \bigoplus \left( \bigoplus_{[\sigma]}
		L^2_{[\sigma]}\right).\]
	Here:
	\begin{itemize}
		\item $\pi$ runs over cuspidal automorphic representations of $G(\BA)$, and 
			$L_\pi^2$ is its $L^2$-completion.
		\item $[\sigma]$ runs over the equivalent classes of unitary automorphic characters $\sigma: F^\times \bs \BA^\times \ra \BC^\times$
	under the action $\sigma \mapsto \sigma^{-1}$ of the non-trivial element in the Weyl group $W$, and the action $\sigma \mapsto \sigma_\lambda = \sigma |\cdot|^\lambda$ $(\lambda \in \BC/(\frac{2\pi \sqrt{-1}}{\log q}) \BZ)$ of unramified  characters.
	The space $L^2_{[\sigma]}$ consists of Eisenstein series associated to the induced representations
	$I_P^G(\sigma_\lambda)$ with $\lambda \in \BC/(\frac{2\pi \sqrt{-1}}{\log q})\BZ$. 
	Here $P\subset \PGL_2$ is the parabolic subgroup consisting of upper-triangular matrices, and we view the character 
	$\sigma_\lambda$ as a representation on the Levi subgroup of $P$, that is, the subgroup of diagonal matrices. 
	See Section \ref{sec:spec-decomp} for the precise definition of $L^2_{[\sigma]}$ in general situation.
	\end{itemize}
	\item {\em Harder's theorem on the finiteness of cuspidal representations}: for any open compact subgroup $K$ of $G(\BA)$, the space $L_0^2(G(F) \bs G(\BA)/K)$ is of finite dimension (see Corollary \ref{cor of Harder} in \S \ref{sec:spec-decomp}). 
	\item {\em The Satake isomorphism}:
		for any place $v\notin S$ (here $S$ is the set of places of $F$ given in Section \ref{sec: intro}), consider the
		trace map
		\[\CH_v \lra C\left(\wh{G(F_v)}_\mathrm{un}\right), \quad f \mapsto \left(\pi \mapsto \tr (\pi(f))\right)\]
		where $\CH_v$ is the spherical Hecke algebra at $v$, $\wh{G(F_v)}_\mathrm{un}$ is the set of equivalent classes of unramified representations of $G(F_v)$ with the Fell topology, and $C\left(\wh{G(F_v)}_\mathrm{un}\right)$ is the space of continuous functions on $\wh{G(F_v)}_\mathrm{un}$. 
		The trace map factors through the Satake isomorphism (see Section \ref{sec: proof})
		\[\CS: \CH_v \stackrel{\sim}{\lra} \BC\left[T,T^{-1}\right]^W, \]
		so that if $\pi$ is a
		subquotient of $I_P^G(|\cdot|_v^\lambda)$, one has
		\[ \tr (\pi(f)) = (\CS f)(q_v^{\lambda},q_v^{-\lambda})\]
		for any $f \in \CH_v$.
		Here, $q_v$ is the cardinality of the residue field of $F_v$ and $|\cdot|_v$ is the normalized abstract value on $F_v$ which
		maps uniformizers to $q_v^{-1}$.
\end{itemize}

\paragraph{\bf Step 1: Killing the continuous spectrum}
We now apply the trick of Lindenstrauss-Venkatesh \cite{LV} for the function field case, which 
is based on the strong relation of an Eisenstein series at some different places.

Let $\pi = \otimes_v \pi_v$ be an irreducible admissible representation of $G(\BA)$. Let $K$ be an open compact 
subgroup of $G(\BA)$ such that $\pi^K$ is nonzero. Let $S$ be a finite place of $F$ such that $K$ is maximal outside $S$.
Assume $\pi$ is not $(G,S)$-CAP. 

There are only finitely many classes of characters $[\sigma]$, so that we may also assume that these $[\sigma]$'s are all unramified outside $S$. 
Here the finiteness comes from the finiteness of the divisor class number of $F$.
For higher rank groups, we need Harder's theorem on finiteness of cuspidal representations (see Theorem \ref{Harder's thm}).

Fix a pair of places
$$\mathbf{v_\infty} = \{v_{\infty,1}, v_{\infty,2}\}$$ 
of $F$ disjoint with $S$, such that 
\begin{equation}\label{isom of poly rings 1}
\BC[q_{\mathbf{v_\infty}}^{\pm \mathbf{\lambda}}]:=\BC[q_{v_{\infty,1}}^{\pm \lambda}, q_{v_{\infty,2}}^{\pm \lambda}] = \BC[q^{\pm \lambda}].
\end{equation}
For this, one may just consider two places $v_{i,\infty}$ ($i=1,2$) with coprime degrees.

Fix a  unitary representation $\sigma$ unramified outside $S$.
For simplicity, we introduce the following notation related to 
$\mathbf{v_\infty}$:
\begin{itemize}
\item[$\Diamond$] $X_i=\BC \big/\frac{2\pi \sqrt{-1}}{\log q_{\infty ,i}}\BZ$ ($i=1,2$), and $X=\BC \big/\frac{2\pi \sqrt{-1}}{\log q}\BZ$.
The Weyl group $W$ acts on $X_1$ and $X_2$ respectively 
so that $W \times W$ acts on $X_1 \times X_2$.
\item[$\Diamond$] Write $W=\langle 1, w\rangle$. 
Elements in $W\times W$ are indexed by $\wt w_1=(1,1)$, $\wt w_2=(w,1)$, $\wt w_3=(1,w)$ and $\wt w_4=(w,w)$.
\item[$\Diamond$] fix $\alpha_{\mathbf{v}_\infty}$ and $\beta_{\mathbf{v}_\infty}\in  X_1\times X_2$ to be such that $\pi_{\mathbf{v}_\infty}=I(|\cdot |_{\mathbf{v}_\infty}^{\alpha_{\mathbf{v}_\infty}})$ and $I(\sigma)_{\mathbf{v}_\infty}=I(|\cdot |_{\mathbf{v}_\infty}^{\beta_{\mathbf{v}_\infty}})$. Here and in the rest of this section, we denote $I(\cdot)=I_P^G(\cdot)$ for simplicity.
\end{itemize}
Now, denote
$\lambda_\infty := \lambda_\infty^{(i)} = 
\alpha_{\mathbf{v}_\infty}^{\wt w_i} - \beta_{\mathbf{v}_\infty} \in 
X_1 \times X_2$, and consider the set of places
$$\mathbf{v}_f = \{v_1,v_2,v_3,v_4\},$$ 
disjoint with $S\cup \mathbf{v}_\infty$, such that for each $i=1,\cdots,4$:
	\begin{enumerate}
		\item[(i)] if $\lambda_\infty \in \Delta$, 
		we take $v_i$ to be such that $\pi_{v_i}$ is not a
		subquotient of  $I(\sigma_{\lambda^\sharp_\infty})_{v_i}.$ 		
		Here $\Delta$ denotes the image of the diagonal map $X\lra X_1\times X_2$ given by $\lambda\mapsto (\lambda, \lambda)$, $\lambda^\sharp_\infty$ denotes any lifting of $\lambda_\infty$ in $X$ via the diagonal map, and $\alpha_{\mathbf{v}_\infty}^{\wt w_i}$ denotes the image of $\alpha_{\mathbf{v}_\infty}$ under
		the action of $\wt{w_i}$.
		\item[(ii)] if $\lambda_\infty \not\in \Delta$, we take $v_i$ to be such that
$\pi_{v_i} \not= I(\sigma)_{v_i}$.
	\end{enumerate}	
Note that the above $v_i$'s exist since $\pi$ is not $(G,S)$-CAP.

For each $i$, we take a generator $T_i \in \BC[q_{v_i}^{\pm \lambda}]^W\simeq \CH_{v_i}$. 
We also fix $\alpha_i$ and $\beta_i\in X_{v_i}$ such that $\pi_{v_i}=I(|\cdot|_{v_i}^{\alpha_i})$ and $I(\sigma)_{v_i}=I(|\cdot|_{v_i}^{\beta_i})$, respectively.
One has the following claim: 
\begin{claim}\label{claim T_i}
Denote by $\underline\lambda$ the multi-variable $(\lambda_1,\lambda_2)$. For each $i=1,\ldots, 4$, there exists $T_{i,\infty}\in \BC[q_{\mathbf{v_\infty}}^{\pm \mathbf{\underline\lambda}}]:=\BC[q_{v_{1,\infty}}^{\pm \lambda_1}, q_{v_{2,\infty}}^{\pm\lambda_2}]$ such that
\[T_{i,\infty}(\beta_{\mathbf{v}_\infty}+\lambda) = 
T_i(\beta_i+\lambda)\]
for all $\lambda\in X$,
and 
\[T_{i,\infty}(\alpha_{\mathbf{v}_\infty}^{\wt w_i})
\not= T_i(\alpha_i).\]
\end{claim}
\begin{proof}
Assume first $\lambda_\infty = \alpha_{\mathbf{v}_\infty}^{\wt w_i} - \beta_{\mathbf{v}_\infty}   \in \Delta$. In this case, recall that we take
$v_i$ such that
$\pi_{v_i}$ is not a subquotient
of $I(\sigma_{\lambda^\sharp_\infty})_{v_i}$.
By \eqref{isom of poly rings 1}, there exsits $T_{i,\infty}$ satisfying 
\[T_{i,\infty}(\beta_{\mathbf{v}_\infty}+\lambda) = T_i(\beta_i+\lambda)\]
for all $\lambda\in X$.
By construction, one has
\[T_{i,\infty}(\alpha_{\mathbf{v}_\infty}^{\wt w_i}) = T_{i,\infty}(\beta_{\mathbf{v}_\infty} + \lambda^\sharp_\infty) = T_i(\beta_i+\lambda^\sharp_\infty)
\not= T_i(\alpha_i),\]
 as $T_i$ is a generator in $\BC[q_{v_i}^{\pm \lambda}]^W$.

Now assume $\lambda_\infty := \alpha_{\mathbf{v}_\infty}^{\wt w_i} - \beta_{\mathbf{v}_\infty}  \not\in \Delta$. In this case, recall that we take $v_i$ such that
$\pi_{v_i} \not= I(\sigma)_{v_i}$.
Since we can find $T_{i,\infty}$ such that $T_{i,\infty}(\beta_{\mathbf{v}_\infty}+\lambda) = 
T_i(\beta_i+\lambda)$ for all $\lambda$, 
it is enough to find $T_{i,\infty}$ satisfying moreover 
\[T_{i,\infty}(\alpha_{\mathbf{v}_\infty}^{\wt w_i}) = T_i(\beta_i),\]
and this can be reduced to find $T_{i,\infty}$ such that
$$T_{i,\infty}(\alpha_{\mathbf{v}_\infty}^{\wt w_i}) \not= 0,$$
and $T_{i,\infty}(\beta_{\mathbf{v}_\infty} +\lambda) = 0$ for all $\lambda$.
Since we assume $\alpha_{\mathbf{v}_\infty}^{\wt w_i} - \beta_{\mathbf{v}_\infty}  \not\in \Delta$, such $T_{i,\infty}$ exists since $\BC[q_{\mathbf{v_\infty}}^{\pm \underline{\lambda}}]$ is the coordinate ring of $X_1\times X_2$.
\end{proof}

We construct the following ($4\times 4$)-matrix:
\[
\left[\begin{matrix}
T_{1,1} = T_1 - T_{1,\infty} & T_{1,2} = T_1 - T_{1,\infty}^{\wt w_2} & 
T_{1,3} = T_1 - T_{1,\infty}^{\wt w_3} & T_{1,4} = T_1 - T_{1,\infty}^{\wt w_4} \\
&&&\\
T_{2,1} = T_2 - T_{2,\infty} & T_{2,2} = T_2 - T_{2,\infty}^{\wt w_2} & 
T_{2,3} = T_2 - T_{2,\infty}^{\wt w_3} & T_{2,4} = T_2 - T_{2,\infty}^{\wt w_4} \\
&&&\\
T_{3,1} = T_3 - T_{3,\infty} & T_{3, 2} = T_3 - T_{3,\infty}^{\wt w_2} & 
T_{3,3} = T_3 - T_{3,\infty}^{\wt w_3} & T_{3,4} = T_3 - T_{3,\infty}^{\wt w_4} \\
&&&\\
T_{4,1} = T_4 - T_{4,\infty} & T_{4,2} = T_4 - T_{4,\infty}^{\wt w_2} &
T_{4,3} = T_4 - T_{4,\infty}^{\wt w_3} & T_{4,4} = T_4 - T_{4,\infty}^{\wt w_4}
\end{matrix}\right].\]
In the above, $T_{i,\infty}^{\wt w_j} := 
T_{i,\infty} \circ \wt w_j$.
Any element in the first column of the above matrix kills the continuous spectrum $I(\sigma_\lambda)$ for all $\lambda$ by Claim \ref{claim T_i}. 
Also, for the elements in the diagonal of the above matrix, their values $(\alpha_i,\mathbf{\alpha_{\mathbf{v}_\infty}})$ are all non-zero.
Denote multi-variables $\underline\lambda_f=(\lambda_1, \lambda_2, \lambda_3, \lambda_4)$, and also $\underline\lambda_\infty=(\lambda_{1,\infty}, \lambda_{2,\infty})$.
Then there exists constants $C_1,\ldots,C_4
		\in \BC$ such that for each $w$, if we set (the weighted sum of columns)
		\[T_{k} = \sum_{i=1}^4 C_i\cdot  T_{i,k} \in \BC\left[q_{\mathbf{v}_f}^{\pm \underline\lambda_f},
		q_{\mathbf{v}_\infty}^{\pm \underline\lambda_\infty}\right],\]
		then all these $T_k$'s ($k=1,\ldots ,4$) are non-zero at 
		$(\alpha_{\mathbf{v}_f}, \alpha_{\mathbf{v}_\infty})$, here $\alpha_{\mathbf{v}_f}=(\alpha_1,\alpha_2,\alpha_3,\alpha_4)$. Moreover
		\[T = \prod_{k=1}^4 T_{k} \in \BC\left[q_{\mathbf{v}_f}^{\pm \underline\lambda_f},
		q_{\mathbf{v}_\infty}^{\pm \underline\lambda_\infty}\right]^{\underline{W}}.\]
In the above notation, $\underline{W} = \prod_{v \in \mathbf{v}_\infty\cup\mathbf{v}_f} W$ acts on each variable respectivly.

Then it is clear that such $T$ annihilates the the continuous spectrum $I_P^G(\sigma_\lambda)$ for all $\lambda$'s but preserves $\pi$. 
As there are finitely many $[\sigma]$'s, a finite product of such Hecke elements $T$ kills the space orthogonal to $L^2_0(G(F) \bs G(\BA)/K)$ in $L^2(G(F) \bs G(\BA)/K)$, but does not kill $\pi$.

\paragraph{\bf Step 2: Isolating $\pi$}
By Harder's theorem, there are only finitely many cuspidal representations in the cuspidal spectrum $L^2_0(G(F) \bs G(\BA)/K)$. 
Denote by $\pi_1,\ldots,\pi_n$ the cuspidal representations which are not nearly equivalent to $\pi$. 
In particular, for $\pi_1$, there is a place $v_1$ of $F$ outside the union of $S$ and $\cup_{[\sigma]}S_\sigma$, such that $\pi_{1,v_1} \not\cong \pi_{v_1}$. Here $S_\sigma$ is a finite set of places such that the Hecke algebra element used to kill the Eisenstein part $L^2_{[\sigma]}$ above lies in $\CH_{S_\sigma}$, and $[\sigma]$ runs over all such equivalence classes. 
Equivalently, we have $\tr(\pi_{1,v_1}) \not= \tr(\pi_{v_1})$, and hence the Hecke algebra element $\left[T_{v_1} - \tr(\pi_{1,v_1}(T_{v_1}))\right]$ kills $\pi_1$, but does not kill $\pi$. 
Here $T_{v_1}$ is a generator in the Hecke algebra $\CH_{v_1}$.
Continue this procedure for $\pi_2, \ldots, \pi_n$, we can construct a Hecke algebra element which kills all the cuspidal representations not nearly equivalent to $\pi$ in the spectrum, but does not kill $\pi$. 
This finishes the proof.

\begin{remark}\label{remark 2.1} 
In \cite{YZ}, the authors kill the continuous part in the case that $G = \PGL_2$ and $K$ is maximal, by employing the so-called {\it Eisenstein ideal} in Hecke algebra. 
Take $S=\emptyset$, and let $\CH = \otimes_v \CH_v$ be the spherical Hecke algebra of $G$ with respect to $K$. 
Consider the ring homomorphisms (See \cite[(4.1)]{YZ})
\[a_{\Eis}: \CH \stackrel{\CS}{\lra} \BC[A(\BA)/A(\BO)] \lra \BC[A(F) \bs A(\BA)/A(\BO)],\]
where $A$ is the diagonal subgroup of $G$, and $\BO = \prod_v \CO_v$. 
The image of $a_\Eis$ can be described clearly, which is the subspace of $\BC[A(F) \bs A(\BA) /A(\BO)]$ invariant under the involution from the Weyl group $W$ of $G$ (see \cite[Lemma 4.2 (2)]{YZ}). 
The Eisenstein ideal $I_\Eis$ of $\CH$ is then defined to be the kernel of $a_\Eis$. 
By the spectral decomposition of $L^2(G(F) \bs G(\BA) / K)$, and the fact that the characters of Eisenstein series factor through $a_\Eis$ (that is,
$\tr (I_P^G \chi) = \chi \circ a_\Eis$ for any unramified Hecke character $\chi$ on $A(\BA)$), any element in $I_\Eis$ kills the continuous spectrum, and vice versa. 
On the other hand, the ideal $I_\Eis$ is large enough in the sense that for any cuspidal automorphic representation $\pi$ of $G(\BA)$ which is unramified everywhere, there exists an element $f \in I_\Eis$ such that $\tr (\pi(f)) \not= 0$. 
In fact, if not, $I_\Eis$ will be contained in the kernel of $\tr \pi$ so that $\tr \pi$ factors though $a_\Eis$. 
Since the image of $a_{\Eis}$ is the $W$-invariant subspace of
$\BC[A(F)\bs A(\BA)/A(\CO)]$, $\tr \pi$ is given by a ($W$-orbit of) Hecke character $\chi$ on $A(\BA)/A(\CO)$. 
This is impossible since $\pi$ is not $(G,\emptyset)$-CAP. 
Moreover, based on the above, one can obtain an element $\mu \in I_\Eis$ satisfying the conditions in Theorem \ref{iso-cusp}, by applying Harder's theorem.

Comparing with the strategy above, for a given cuspidal automorphic representation $\pi$, in this section we construct an explicit element $\mu_\pi\in I_\Eis$ isolating $\pi$ from $L^2(G(F) \bs G(\BA)/K)$. 
In particular, $\mu_\pi$ depends on $\pi$, while $I_\Eis$ does not.
\end{remark}

\begin{remark}\label{remark 2.2}
We discuss a possible generalization of the above Eisenstein ideal for a general reductive group $G$ and any level $K \subset
G(\BA)$. Let $\omega$ be a unitary Hecke character on the center $Z$ of $G$ and 
assume $\omega$ is invariant under $Z(\BA) \cap K$. Let
$S$ be a finite set of places such that $K^{(S)}$ is maximal. For each standard Levi subgroup $M$ of $G$ (after fixing a Borel subgroup of $G$), denote by 
$\CH_M^{(S)} = \otimes_{v \not\in S} \CH_{M,v}$ the spherical Hecke algebra of $M$ outside $S$. Similar to the Satake transform (for the minimal Levi),
for each $v \not\in S$, consider the map
\[\CS_{M,v}: \CH_v \lra \CH_{M,v}, \quad f_v \mapsto \left(m\mapsto \delta_{P_M}(m)^{1/2} \int_{N(F_v)} f_v(m n)\,\mathrm{d}n\right).\]
Then for each irreducible unramified representation $I_{P}^G \sigma_v$ of $G(F_v)$, one has
\[\tr (I_P^G(\sigma_v)(f_v)) = \tr (\sigma_v(\CS_{M,v}(f_v))), \quad (f_v \in \CH_v).\]
Denote by $\CS_M = \otimes_{v \not\in S} \CS_{M,v}$, and consider the following map
\[a_M: \CH^{(S)} \stackrel{\CS_M}{\lra} \CH_M^{(S)} \lra \End_{\CH_M^{(S)}}\left( \CA_\cusp(M(\BA)/K \cap M(\BA))_\omega \right).\]
Here, $\CA_\cusp(M(\BA)/K \cap M(\BA))_\omega$ is the space of cusp forms $\varphi$ on $M(\BA)$ 
such that
\[\varphi(zgk) = \omega(z)\varphi(g), \quad \forall z \in Z(\BA),g \in G(\BA), k \in K \cap M(\BA).\]
Denote by $I_M$ the kernel
of $a_M$. Consider the following ideal of $\CH^{(S)}$:
\[I_\Eis = \bigcap_M I_M.\]
Then by the spectral decomposition of $L^2(G(F) \bs G(\BA)/K)_\omega$, 
any element in $I_\Eis$ kills the continuous spectrum, and vice versa. 
On the other hand, one needs to know that the ideal $I_\Eis$ is large enough in the sense that for any irreducible admissible representation $\pi$ on $G(\BA)$ with $\pi^K \not=0$ which is not
$(G,S)$-CAP, there exists an element $f \in I_\Eis$ such that $\pi(f) \not= 0$. 
One may prove this by studying the image of $a_M$ for each $M$ as in Remark \ref{remark 2.1}, but it seems more involved.  
However, the property that $I_\Eis$ is large enough will follow from Theorem \ref{iso-cusp} immediately, which ensures that there is a $\mu \in I_\Eis$ such that $\pi(\mu) = 1$.
\end{remark}

\section{Spectral decomposition along the cuspidal data}\label{sec:spec-decomp}

In this section, we recall the spectral decomposition of $L^2(G(F) \bs G(\BA))_\omega$ along the cuspidal supports in the case of function field, following \cite{MW}. 

For convenience, we list some notation first, which will be used in the remaining parts of this note.
\begin{itemize}
	\item Let $P_0$ be a fixed minimal parabolic subgroup of $G$ defined over $F$, with Levi subgroup $M_0$.
		A subgroup $M$ of $G$ is called a standard Levi subgroup if there exists a parabolic subgroup of $G$ containing $P_0$, of which $M$ is the unique Levi subgroup containing $M_0$. 
	\item Let $T_0$ be the maximal split torus in the center of $M_0$. For any standard Levi subgroup of $G$, let $T_M$ be the maximal split torus in the center of $M$, which is contained in $T_0$.	
	\item Fix a maximal compact subgroup $K_0\subset G(\BA)$ such that 
	\begin{itemize}
	\item $G(\BA)=P_0(\BA)K_0$;
	\item for every standard parabolic subgroup $P=MU$, $P(\BA)\cap K_0=(M(\BA)\cap K_0)(U(\BA)\cap K_0)$, and $M(\BA)\cap K_0$ is a maximal compact subgroup of $M(\BA)$.
	\end{itemize}
	The choice of $K_0$ fixes a choice of maximal compact subgroup of $M(\BA)$ for every standard Levi $M$.
	\item Denote by $\mathrm{Rat}(M)$ the group of rational characters of $M$. Then denote $\Re({\fa_M^*})=\mathrm{Rat}(M)\otimes_\BZ \BR$, and $\fa_M^* = \mathrm{Rat}(M) \otimes_\BZ \BC$.
	\item For $\chi \in \mathrm{Rat}(M)$, denote by $|\chi|$ the continuous character on $M(\BA)$
	      given by 
	      \[|\chi|(m) = \prod_v |\chi_v(m_v)|_v, \quad (m = (m_v)_v \in M(\BA))\] 
	      where $\chi_v: M(F_v) \ra F_v^\times$ is the algebraic character induced by $\chi$. Then define 
	      \[M(\BA)^1 = \bigcap_{\chi \in \Rat(M)} \Ker |\chi|.\]
	\item Denote by $X_M$ the group of characters on $M(\BA)^1  \bs M(\BA)$, which can be realized
		as a quotient of $\fa_M^*$. 
		In fact, let $\chi_1,\ldots,\chi_r$ be a $\BZ$-basis of $\mathrm{Rat}(M)$, the map 
			\[j : M^1(\BA)\bs M(\BA)\lra \left(q^\BZ\right)^r,\quad m\mapsto \left(|\chi_1|(m), \ldots , |\chi_r|(m)\right)\]
			defines a topological group isomorphism onto its image, which is a subgroup of $(q^\BZ)^r$ with finite index.
		Then 
		\[\kappa: \fa_M^* \lra X_M, \quad \chi_i \mapsto |\chi_i|\]
		is a surjective morphism of groups, and the kernel of $\kappa$ is of the form 
		$(\frac{2\pi i}{\log q})L$, where $L$ is a lattice of $\Rat(M) \otimes_\BZ \BQ$. We also denote $\Re (X_M)=\kappa(\Re (\fa_M^*))$, and $\kappa$ induces an isomorphism $\Re (\fa_M^*)\simeq \Re (X_M)$.
	\item Denote by $X_M^G$ the subgroup of $X_M$ with characters trivial on $Z(\BA)$ (recall that $Z=Z_G$, the center of $G$). In particular, 
		there is a perfect pairing 
		\begin{equation}\label{pairing for X_M^G}X_M^G \times M(\BA)^1 Z(\BA) \bs M(\BA) \lra \BC^\times.\end{equation}
	\item For standard Levi subgroups $M\subset M'$ of $G$, denote by $\Re ((\fa_M^{M'})^*)$ the real vector subspace of $\Re (\fa_M^*)$ generated by $R(T_M, M')$, the set of roots (see \cite[\S I.1.6]{MW}) of $M'$ relative to $T_M$.
	Identifying $\Re(\fa_{M'}^*)$ with a real vector subspace of $\Re (\fa_M^*)$ by restriction, we have 
	\[\Re (\fa_M^*)= \Re (\fa_{M'}^*)\oplus  \Re((\fa_M^{M'})^*).\]
	Moreover, the elements of $\Re((\fa_M^{M'})^*)$ can be identified with the elements of $\Re (X_M)$ which are trivial on the center of $M'(\BA)$.
	After tensor product by $\BC$, one also has the decomposition
	\begin{equation}\label{decomposition of a^*_M}
		\fa_M^*=\fa_{M'}^*\oplus (\fa_M^{M'})^*,
	\end{equation}
	where $ (\fa_M^{M'})^*=\Re ((\fa_M^{M'})^*)\otimes _\BR \BC$.
	\item For a compact open subgroup $K\subset G(\BA)$ such that $G(\BA)=P(\BA)K$, one defines a map 
	$$m_P: G(\BA)\lra M^1(\BA)\bs M(\BA)$$
	by $m_P(g)=M^1 m$ if $g=muk$ with $u\in U(\BA)$, $m\in M(\BA)$ and $k\in K$.

\end{itemize}

We recall some notions on automorphic forms and automorphic representations.
Let $P=MU$ be a standard parabolic subgroup. 
We call a smooth (locally constant) function 
\begin{equation}\label{auto form}\varphi: U(\BA)M(F)\bs G(\BA)\lra \BC\end{equation}
an {\it automorphic form} if 
\begin{itemize}
	\item[(i)] $\varphi$ is of moderate growth;
	\item[(ii)] $\varphi$ is $K_0$-finite;
	\item[(iii)] $\varphi$ is $\fZ(G(F_v))$-finite for any place $v$ of $F$. Here, $\fZ(G(F_v))$ is the Bernstein
		center (see \cite{BD}) of $G(F_v)$. 
\end{itemize}
We denote the space of all such automorphic forms by $\CA(U(\BA)M(F) \bs G(\BA))$. 
For a unitary automorphic character $\omega: Z(F)\bs Z(\BA) \lra \BC^\times$, we also denote by $\CA(U(\BA)M(F) \bs G(\BA))_\omega$ the automorphic forms $\varphi$ with central character $\omega$, i.e. $\varphi(zg)=\omega(z)\varphi(g)$ for all $z\in Z(\BA)$.
We say $\varphi$ is {\it cuspidal} if for all parabolic subgroups $P'$ with $P_0\subset P'\subsetneq P$, its constant term along $P'$ is zero.
The space of cuspidal automorphic forms on $U(\BA)M(F)\bs G(\BA)$ is denoted by $\CA_0(U(\BA)M(F)\bs G(\BA))$.

Moreover, for any $k\in K_0$, we define $\varphi_k: M(F)\bs M(\BA)\lra \BC$ by 
\[\varphi_k(m) = m^{-\rho_P} \varphi(mk),\]
where $\rho_P$ is the half-sum of roots of $M$ in the Lie algebra of $U$.
Then a smooth function \eqref{auto form} is an automorphic form if it is $K_0$-finite and for all $k\in K_0$, $\varphi_k$ is an automorphic form on $M(F)\bs M(\BA)$ (\cite[\S I.2.17]{MW}).

The spectral decomposition is given by Eisenstein series associated to different cuspidal data. 
We set some more notation:
\begin{itemize}
\item Denote by $\Pi_0(M(\BA))$ the set of cuspidal automorphic representations $\sigma$ of $M(\BA)$, i.e., the set of equivalence classes of irreducible subquotients of the space of cusp forms $\CA_0(M(F)\bs M(\BA))$. 
\item For any unitary automorphic character $\omega: Z(F)\bs Z(\BA)\lra \BC^\times$, let $\Omega_M(\omega)$ be the set of unitary automorphic characters $\omega_M: Z_M(F)\bs Z_M(\BA)\lra \BC^\times$ such that $\omega_M|_{Z(\BA)}=\omega$.
\item Denote by $\Pi_0(M(\BA))_\omega$ the subspace of $\Pi_0(M(\BA))$ consisting of cuspidal automorphic representations with central character $\omega_M\in \Omega_M(\omega)$.
\item For $\sigma \in \Pi_0(M(\BA))_\omega$, denote by $\CA(M,\sigma)$ the subspace consisting of $$\varphi \in \CA(U(\BA)M(F) \bs G(\BA))_\omega$$ such that $\varphi_k \in \CA(M(F) \bs M(\BA))_{\sigma}$ for all $k \in K_0$, here $\CA(M(F) \bs M(\BA))_{\sigma}$ is the isotypic submodule of $\sigma$ in $\CA(M(F) \bs M(\BA))$.
\end{itemize}
The group $X_M^G$ acts on the space $\Pi_0(M(\BA))_\omega$ via 
$$\sigma\lto \sigma_\lambda =  \sigma\otimes \lambda$$ 
       with $\lambda \in X_M^G$ and $\sigma \in \Pi_0(M(\BA))_\omega$.
We say $\sigma$ is equivalent to $\sigma'$ if there exisits $\lambda\in X_M^G$ such that $\sigma_\lambda\simeq \sigma'$, and denote such an equivalent class by $\fP$.
A cuspidal datum (of central character $\omega$) is a pair $(M,\fP)$ where $M$ is a standard Levi of $G$, and $\fP$ is an equivalence class of $\sigma\in \Pi_{0}(M(\BA))_\omega$ as above.
Two cuspidal data $(M,\fP)$ and $(M',\fP')$ are called equivalent if there exists some $w \in G(F)$ such that 
$w\cdot M = M'$ and $w\cdot\fP = \fP'$. By the Bruhat decomposition, if such $w$ exists, we can suppose it lies in the Weyl group of $G$.

The group $X_M^G$ also acts on the space $\CA(U(\BA)M(F) \bs G(\BA))_\omega$ via
       $$\varphi\lto \varphi_\lambda := \varphi\cdot (\lambda\circ m_P).$$
       Then any $\lambda\in X_M^G$ induces an isomorphism
\[\lambda: \CA(M,\sigma) \stackrel{\sim}{\lra} \CA(M,\sigma_\lambda).\]
For any $\varphi \in \CA(M,\sigma)$, the Eisenstein series on $G(F) \bs G(\BA)$ associated to $\varphi$ is defined by
\[E(\varphi,\sigma)(g) = \sum_{\gamma \in P(F) \bs G(F)} \varphi(\gamma g),\]
whenever the sum converges.

\begin{prop}[\cite{MW} \S II.1.5, Proposition] 
   There exists an open cone $C_M^G$ in $X_M^G$
   such that for any $\varphi \in \sigma$, if $\lambda \in C_M^G$, 
   then the summation defining $E(\varphi_\lambda,\sigma_\lambda)(g)$ converges absolutely and uniformly when $g$ varies in a compact set. 
   Moreover, one also has
   \[E(\varphi,\sigma) \in \CA(G(F) \bs G(\BA))_\omega\]
   if it is convergent.
\end{prop}

Let $P(X_M^G)$ be the set of Paley-Wiener functions on $X_M^G$, i.e. the image of the Fourier transform (recall \eqref{pairing for X_M^G}) 
\[f \mapsto \wh{f}(\lambda) = \int_{M(\BA)^1 Z(\BA) \bs M(\BA)} f(m)\lambda(m)\,\mathrm{d}m\]
on the space $C_c^\infty(M(\BA)^1Z(\BA) \bs M(\BA))$.
		A section $\Phi: X_M^G \ra \CA(M,\sigma)$ is called a Paley-Wiener section if $\Phi$ is a sum of sections 
of the form
\[X_M^G\ni \lambda \lto \wh f(\lambda)\cdot \varphi\]
for some $\wh f \in P(X_M^G)$ and $\varphi \in \CA(M,\sigma)$.
Denote by $P(M,\sigma)$ the space of all the Paley-Wiener sections on $\CA(M,\sigma)$. 
For any $\Phi \in P(M,\sigma)$, consider the pseudo-Eisenstein series (see \cite[\S II.1.11 and \S II.1.12]{MW})
\[\theta_\Phi(g) = \int_{\lambda\in X_M^G, \Re(\lambda) = \lambda_0} E(\Phi(\lambda)_\lambda,\sigma_\lambda)\,\mathrm{d}\lambda\]
where $\lambda_0$ is an arbitrary element in $\Re(X_M^G)$  which is positive enough.

Let $L^2(G(F)\bs G(\BA))_\omega$ be the space of functions on $G(F)\bs G(\BA)$ with central character $\omega$ and square-integrable modulo the center $Z(\BA)$.
By computing the inner product of two psuedo-Eisenstein series, one obtains the following spectral decomposition result:
\begin{thm}[Spectral decomposition along cuspidal data, \cite{MW} \S II.2.4, Proposition] 
	Let $\fX$ be an equivalence class of cuspidal data. Denote by 
	$L^2(G(F) \bs G(\BA))_\fX$ the closed subspace of $L^2(G(F) \bs G(\BA))_\omega$ spanned by 
	$\theta_{\Phi}$ with $\Phi \in P(M,\sigma)$ where $(M,\sigma)$ is an arbitrary representative of $\fX$. 
	Then
	\[L^2(G(F) \bs G(\BA))_\omega = \bigoplus_{\fX} L^2(G(F) \bs G(\BA))_\fX.\]
\end{thm}

We also need some finiteness property for the spectral decomposition in our proof later. 
The following theorem is due to Harder:

\begin{thm}[\cite{Harder}, Corollary 1.2.3]\label{Harder's thm} 
Let $G$ be a reductive group over $F$ and $\omega$ be a 
	unitary character of $Z(F)\bs Z(\BA)$. Then for any open compact subgroup $K$
	of $G(\BA)$, the vector space $L_0^2(G(F) \bs G(\BA)/K)_\omega$ is of finite dimension.
\end{thm}

\begin{cor}\label{cor of Harder}
	Let $G$ be a reductive group over $F$ and $\omega$ be a 
	unitary character of $Z(F)\bs Z(\BA)$. Let $K$ be an open compact subgroup of $G(\BA)$.
	Then there are only finitely many cuspidal data occurring in the spectral decomposition of
	$L^2(G(F) \bs G(\BA) /K)_\omega$.
\end{cor}
\begin{proof} 
   Let $\fX$ be an equivalence class of cuspidal data. 
   Assume that $L^2(G(F) \bs G(\BA)/K)_\fX \not= 0$.
   Then for any $(M,\sigma) \in \fX$, we have $(I_P^G \sigma)^K \not=0$. 
  Using \cite[\S III.2.2, Lemme]{Renard}, one sees that there exists an open compact subgroup $K_M$ of $M(\BA)$ depending on $K$ such that $\sigma^{K_M} \not =0$ for any $\sigma$ with $(I_P^G \sigma)^K \not=0$. 
   We claim that by modifying $\sigma$ to $\sigma_\lambda$ with $\lambda \in X_M^G$,
   the central character of $\sigma$ belongs to a finite set. 
   Hence, by Harder's theorem (Theorem \ref{Harder's thm}), there are only finitely many such $\sigma$'s.

   To prove the claim, consider the
   set $\Sigma$ of characters $\omega_M: Z_M(F) \bs Z_M(\BA)/K_{Z_M} \ra \BC^\times$ 
   with $\omega_M|_{Z(\BA)} = \omega$, here $K_{Z_M}$ is a fixed open compact subgroup
   of $Z_M(\BA)$. 
   The group of characters of the quotient $Z_G(\BA)Z_M(\BA)^1 \bs Z_M(\BA)$ acts on $\Sigma$. 
   The claim is then equivalent to saying that the number of the orbits of $\Sigma$ under this action is finite. 
   To see the finiteness, let $T$ be the torus $Z_M/Z_G$ over $F$ and $K_T$ the image of 
   $K_{Z_M}$ in $T(\BA)$.
   Consider the exact sequence
   \[1 \lra Z_G(F) \bs Z_G(\BA) / (K_{Z_M} \cap Z_G(\BA)) \lra Z_M(F) \bs Z_M(\BA) /K_{Z_M}
   \lra T(F) \bs T(\BA) /K_T \lra 1.\]
   We may write $T = T_s\cdot T_0$ with $T_s \cap T_0$ being finite, where 
   $T_s$ is a split torus and $T_0$ an anisotropic torus. 
   For the anisotropic part, the quotient $T_0(F) \bs T_0(\BA)$ is compact. 
   And for the split part, if we denote its rank by $d$, then 
   $T_s(F) \bs T_s(\BA) \cong (F^\times \bs \BA^1)^d \times \BZ^d$
   with $(F^\times \bs \BA^1)^d$ compact. 
   Note that we may modify $\omega_M$ by a character on 
   $Z_G(\BA)Z_M(\BA)^1 \bs Z_M(\BA)$ such that $\omega_M$ is trivial on $\BZ^d$. 
   Therefore, there must be finite number of such orbits by the exact sequence above.

\end{proof}

\section{Proof of Theorem \ref{iso-cusp}}\label{sec: proof}

We prove Theorem \ref{iso-cusp} in general case in this section.
We recall some basics on unramified representations and the Satake isomorphism at first, and the basic reference is \cite{Car}.

We keep the notation used in Section \ref{sec: intro} and Section \ref{sec:spec-decomp}.
Recall that $T_0$ is the maximal split torus of $M_0$. 
We denote by $d$ the rank of $T_0$, and fix a basis of $\mathrm{Rat}(T_0)$, say $\chi_1,\ldots,\chi_d$.
Let $v$ be a place of $F$ outside $S$. 
Denote by $\wh{M_0(F_v)}_\un$ the group of unramified characters on $M_0(F_v)$.  
Then we have an isomorphism
\[\left( \BC \big/(\frac{2\pi i}{\log q_v})\BZ \right)^d \stackrel{\sim}{\lra} \wh{M_0(F_v)}_\un,
\quad (\lambda_1,\ldots,\lambda_d) \mapsto |\chi_1|^{\lambda_1} \cdots |\chi_d|^{\lambda_d},\]
so that we may view $\wh{M_0(F_v)}_\un$ as a torus over $\BC$. 
Denote by $\BC\left[\wh{M_0(F_v)}_\un\right]$ the ring of regular functions on $\wh{M_0(F_v)}_\un$. Under the
above isomorphism, one has
\begin{equation}\label{isom 1 in sec 4}
    \BC\left[\wh{M_0(F_v)}_\un\right] \cong \BC\left[q_v^{\lambda_1},q_v^{-\lambda_1},\ldots,q_v^{\lambda_d},q_v^{-\lambda_d}\right]
\end{equation}
Let $v_1$ be another place of $F$ such that $q_{v_1} = q_v^{k}$ for some integer $k$. 
Then by taking $q_{v_1}^{\pm \lambda_i} \mapsto \left(q_v^{\pm \lambda_i}\right)^k$ ($i=1,\ldots,d$), we have an injection
\[\BC\left[\wh{M_0(F_{v_1})}_\un\right] \hookrightarrow \BC\left[\wh{M_0(F_v)}_\un\right]\]
from \eqref{isom 1 in sec 4}.

There is a perfect pairing
\[\wh{M_0(F_v)}_\un \times M_0(F_v)/M_0(\CO_v) \lra \BC^\times.\]
For each $f \in C_c^\infty(M_0(F_v)/M_0(\CO_v))$, one considers its Fourier transform
\[\wh{f}(\chi) = \int_{M_0(F_v)/M_0(\CO_v)} f(m)\chi(m)\, \mathrm{d}m, \quad \left(\chi\in \wh{M_0(F_{v})}_\un\right)\]
which gives an isomorphism
\[\wh{}: C_c^\infty(M_0(F_v)/M_0(\CO_v)) \stackrel{\sim}{\lra} \BC\left[\wh{M_0(F_v)}_\un\right].\]

Denote by $\wh{G(F_v)}_\un$ the set of irreducible unramified representations of $G(F_v)$, i.e. 
the irreducible smooth representations $\pi_v$ of $G(F_v)$ with the non-zero invariant subspace $\pi_v^{K_{0,v}}$. 
For any $\chi_v \in \wh{M_0(F_v)}_\un$, there is a unique subquotient of $I_{P_0}^G(\chi_v)$ which is an irreducible unramified representation of $G(F_v)$. 
This in fact gives an isomorphism
\[\wh{M_0(F_v)}_\un/W \stackrel{\sim}{\lra} \wh{G(F_v)}_\un,\]
where $W = N_{G(F)}(M_0(F))/M_0(F)$ is the Weyl group of $G$. Conversely, for an irreducible unramified representation
$\pi_v \in \wh{G(F_v)}_\un$, we denote by $\chi_{\pi_v} \in \wh{M_0(F_v)}_\un/W$ the $W$-orbit of the unramified character corresponding to $\pi_v$ as above.

For each $\pi_v \in \wh{G(F_v)}_\un$, 
the spherical Hecke algebra $\CH_v = C_c^\infty(K_{0,v} \bs G(F_v)/K_{0,v})$ acts on the spherical line $\pi_v^{K_{0,v}}$ of $\pi_v$, which gives a map
\[\tr: \CH_v \lra C\left(\wh{G(F_v)}_\un\right), \quad f \mapsto \left(\pi_v \mapsto \tr (\pi_v(f))\right).\]
Recall that $C\left(\wh{G(F_v)}_\mathrm{un}\right)$ is the space of continuous functions on $\wh{G(F_v)}_\mathrm{un}$.
Consider the Satake isomorphism
\[\CS: \CH_v \stackrel{\sim}{\lra} C_c^\infty(M_0(F_v)/M_0(\CO_v))^W\]
given by
\[(\CS f)(m) = \delta_{P_0}(m)^{1/2}\int_{U_0(F_v)} f(mn)\, \mathrm{d}n. \quad (f \in \CH_v)\] 
Then the composition map
\[\CH_v \stackrel{\tr}{\lra} C\left(\wh{G(F_v)}_\un\right) \stackrel{\sim}{\lra} C\left( \wh{M_0(F_v)}_\un /W \right)\]
factors through the isomorphism
\[\CH_v \stackrel{\CS}{\lra} C_c^\infty\left(M_0(F_v)/M_0(\CO_v)\right)^W \stackrel{\wh{}}{\lra} \BC\left[\wh{M_0(F_v)}_\un\right]^W.\]
In particular, we will view elements in $\CH_v$ as functions on $\wh{M_0(F_v)}_\un$ in the following.

Let $\pi$ be an irreducible admissible representation of $G(\BA)$ with central character $\omega$. 
Let $K = K_S \times K_0^{(S)}$ be an open compact subgroup of $G(\BA)$ such that $\pi^K \not= 0$. In particular, $\pi$ is unramified outside $S$. 
Assume that $\pi$ is not $(G,S)$-CAP. 
Let $\fX = [(M,\sigma)]$ be an equivalence class of cuspidal datum with $M \not= G$ such that $L^2(G(F) \bs G(\BA)/K)_\fX \not= 0$. In particular, $\sigma$ is also unramified outside $S$. 
In the following, we want to construct a Hecke algebra element $\mu_\sigma \in \CH^{(S)}$ such that 
\begin{enumerate}
	\item $R(\mu_\sigma)$ acts on $L^2(G(F) \bs G(\BA) /K)_\fX$ by zero;
	\item $\pi(\mu_\sigma) = 1$. 
\end{enumerate}

\paragraph{\bf Step 1: Killing the continuous spectrum}
Note that the restriction map $\fa_M^* \hookrightarrow \fa_{M_0}^*$
is injective, and we fix a splitting of this injection
\[\ell: \fa_{M_0}^* \lra \fa_M^*.\]

Fix a pair of places
$$\mathbf{v}_\infty^{}=\{v^{}_{\infty, 1}, v^{}_{\infty,2}\},$$ 
disjoint with $S$, such that 
		\begin{equation}\label{isom of poly rings}
		\BC\left[q^{\pm \lambda_1}_{\mathbf{v}_{\infty}^{}},\ldots, q^{\pm \lambda_d}_{\mathbf{v}_{\infty}^{}}\right]=\BC\left[q^{\pm\lambda_1},\ldots,q^{\pm\lambda_d}\right]	
		\end{equation}

Fix a cuspidal automorphic representation $\sigma$ on $M(\BA)$.
Similar to the $\PGL(2)$ case, we denote:
\begin{itemize}
\item[$\Diamond$] $X_i=\BC \big/\frac{2\pi \sqrt{-1}}{\log q_{\infty ,i}}\BZ$ ($i=1,2$), and $X=\BC \big/\frac{2\pi \sqrt{-1}}{\log q}\BZ$. 
The Weyl group $W$ acts on $X_i^d$ $(i =1,2)$ by permutations, so that $W \times W$ acts on $(X_1 \times X_2)^d$.
\item[$\Diamond$] $\alpha_{\mathbf{v}_\infty}=\left((\lambda^{}_{1,1}, \lambda^{}_{2,1}), \ldots,  (\lambda^{}_{1,d}, \lambda^{}_{2,d})\right)\in \left(X_1\times X_2 \right)^d$ 
such that the Satake parameter for $\pi_{v_{\infty, i}}$ is given by $(q_{v_{\infty,i}}^{\lambda^{}_{i,1}}, \ldots, q_{v_{\infty,i}}^{\lambda^{}_{i,d}})$ $(i=1,2)$. 
In other words, for $v_{\infty, i}\in \mathbf{v}_\infty$, one has
$$\tr(\pi_{v_{\infty ,i}}(f))=(\CS f)(\lambda_{i,1}, \ldots, \lambda_{i,d}), \quad f \in \CH_{v_{\infty,i}}.$$ 
We also fix $\beta_{\mathbf{v}_\infty}$ similarly by replacing $\pi$ to $I_{P_M}^G(\sigma)$.	
\end{itemize}
Denote $\lambda_\infty=\lambda_\infty^{(w,w')} = \alpha^{(w,w')}_{\mathbf{v}^{}_\infty}-\beta_{\mathbf{v}^{}_\infty}\in (X_1 \times X_2)^d$.
By the condition that $\pi$ is not a $(G,S)$-CAP representation, one can find the following set places: 
$$\mathbf{v}_f^{}=\{v^{}_{w,w'}\}_{(w,w')\in W\times W},$$
disjoint with $S\cup \mathbf{v}_\infty,$
such that for each $(w,w')\in W\times W$:
	\begin{enumerate}
		\item[(i)] if $\lambda_\infty \in \Delta^d$ 
		as a vector in $\left(X_1\times X_2 \right)^d$, $\pi_{v_{w,w'}^{}}$ is not a subquotient of the following local component of parabolic induced representation: $$I_{P_M}^G\left(\sigma_{\ell(\lambda^\sharp_\infty)}\right)_{v_{w,w'}^{}}.$$ 
		Here, as before, we denote by $\Delta$ the image of the diagonal map $X \lra X_1 \times X_2$,
		and denote $\lambda^\sharp_\infty$ to be any lifting of $\lambda_\infty=\alpha^{(w,w')}_{\mathbf{v}^{}_\infty}-\beta_{\mathbf{v}^{}_\infty}\in \Delta^d$ 
		to $\fa_{M_0}^* \cong \BC^d$.
		\item[(ii)]  if $\lambda_\infty \notin 
		\Delta^d$, $\pi_{v_{w,w'}^{}}$ is not a subquotient of $$I_{P_M}^G(\sigma)_{v_{w,w'}}.$$
	\end{enumerate}

Let $(w,w')\in W\times W$. 
Suppose first that $\lambda_\infty=\alpha^{(w,w')}_{\mathbf{v}^{}_\infty}-\beta_{\mathbf{v}^{}_\infty}\in 
\Delta^d$.  
In this case, we take   
$$T_{w,w'}\in \CH_{v_{w,w'}}\simeq \BC\left[q_{v_{w,w'}}^{\pm\lambda_1},\ldots,q_{v_{w,w'}}^{\pm\lambda_d}\right]^W,$$
indexed by $(w,w')\in W\times W$, such that 
\begin{equation}\label{T_{w,w'}-I}
	T_{w,w'}(\alpha_{v_{w,w'}})\neq T_{w,w'}(\beta_{v_{w,w'}}+\lambda^\sharp_\infty).
\end{equation}
Here $\alpha_{v_{w,w'}}$ is a fixed element in $X_{v_{w,w'}}^d$ 
such that $$\tr(\pi_{v_{\infty ,1}}(f))=(\CS f)(\alpha_{v_{w,w'}})$$
for $f\in \CH_{v_{w,w'}}$. We also fix $\beta_{v_{w,w'}} \in 
X_{v_{w,w'}}^d$ by replacing $\pi$ to $I_{P_M}^G \sigma$.
On the other hand, suppose that 
$\lambda_\infty=\alpha^{(w,w')}_{\mathbf{v}^{}_\infty}-\beta_{\mathbf{v}^{}_\infty}\notin \Delta^d$.  
In this case, we take $T_{w,w'}\in \CH_{v_{w,w'}}$ such that 
\begin{equation}\label{T_{w,w'}-II}
	T_{w,w'}(\alpha_{v_{w,w'}})\neq T_{w,w'}(\beta_{v_{w,w'}}).
\end{equation}

\begin{claim}\label{claim T_{w,w'}}
	Denote the multi-variables $\underline\lambda_i=(\lambda_{1,i},\lambda_{2,i})$ $(i=1,\ldots, d)$. For each $(w,w')\in W\times W$, there exists $T_{w,w',\infty}\in \BC\left[q^{\pm \underline\lambda_1}_{\mathbf{v}_{\infty}^{}},\ldots, q^{\pm \underline\lambda_d}_{\mathbf{v}_{\infty}^{}}\right]$ such that 
	\begin{itemize}
		\item[(1)] $T_{w,w',\infty}(\beta_{\mathbf{v}_\infty}+\lambda)=T_{w,w'}(\beta_{v_{w,w'}}+\lambda)$ for all $\lambda\in X^d$;
		\item[(2)] $T_{w,w',\infty}(\alpha^{(w,w')}_{\mathbf{v}^{}_\infty})\neq T_{w,w'}(\alpha_{v_{w,w'}})$.
	\end{itemize}
\end{claim}
\begin{proof}
	Suppose that $\lambda_\infty=\alpha^{(w,w')}_{\mathbf{v}^{}_\infty}-\beta_{\mathbf{v}^{}_\infty}\in \Delta^d$. 
	By \eqref{isom of poly rings}, we can find $T_{w,w',\infty}$ for any fixed $T_{w,w'}$ satisfying (1) above.
		Moreover, for any $T_{w,w',\infty}$ satisfying (1), it must satisfy (2) by condition \eqref{T_{w,w'}-I} and (i) above.
		
	Suppose that $\lambda_\infty=\alpha^{(w,w')}_{\mathbf{v}^{}_\infty}-\beta_{\mathbf{v}^{}_\infty}\notin \Delta^d$. 	
	Then the existence of $T_{w,w',\infty}$ satisfying (1) and (2) follows from a similar argument as in the proof of Claim \ref{claim T_i}, using \eqref{T_{w,w'}-II} above.
\end{proof}

Granting the above, for any $(w,w')$ and $(\omega,\omega')$ in $W\times W$, we construct the following matrix
\begin{equation}\label{matrix in general}
	\left(T_{(w,w')}^{(\omega,\omega')}:= T_{w,w'}-T_{w,w',\infty}^{(\omega,\omega')}\right)_{(|W|\times |W|)\times (|W|\times |W|)},
\end{equation}
where $T_{w,w',\infty}^{(\omega,\omega')}=T_{w,w',\infty}\circ (\omega,\omega')$.
Note that by Claim \ref{claim T_{w,w'}}, Part (2), in the diagonal of the matrix \eqref{matrix in general}, one has
$$T_{(w,w')}^{(w,w')}(\alpha_{v_{w,w'}}, \alpha_{\mathbf{v}_\infty}):= T_{w,w'}(\alpha_{v_{w,w'}})-T_{w,w',\infty}^{(w,w')}(\alpha_{\mathbf{v}_\infty})\neq 0.$$
For $i=1,\ldots, d$, we denote the multi-variables $\underline\lambda_{i,f}=(\lambda_{i,w,w'})_{(w,w')\in W\times W}$, and also $\underline\lambda_{i,\infty}=(\lambda_{i,1,\infty}, \lambda_{i,2,\infty})$.
It follows that there are constants $C_{(w,w')}^{(\omega,\omega')}\in \BC$ such that 
$$T^{(\omega,\omega')}:=\sum_{(w,w')\in W\times W} C_{(w,w')}^{(\omega,\omega')}\cdot T_{(w,w')}^{(\omega,\omega')}\in \BC\left[q_{\mathbf{v}_f}^{\pm \underline\lambda_{1,f}}, \ldots, q_{\mathbf{v}_f}^{\pm \underline\lambda_{d,f}}, q_{\mathbf{v_\infty}}^{\pm \underline\lambda_{1,\infty}},\ldots, q_{\mathbf{v_\infty}}^{\pm \underline\lambda_{d,\infty}}\right]$$
is non-zero at $(\alpha_{\mathbf{v}_{f}}, \alpha_{\mathbf{v}_\infty})$ for all $(\omega,\omega')\in W\times W$.
Here $\alpha_{\mathbf{v}_f}=(\alpha_{v_{w,w'}})_{(w,w')\in W\times W}$.
Finally, we define 
$$\mu_\sigma:=\prod_{(\omega,\omega')\in W\times W}T^{(\omega,\omega')}\in  \BC\left[q_{\mathbf{v}_f}^{\pm \underline\lambda_{1,f}}, \ldots, q_{\mathbf{v}_f}^{\pm \underline\lambda_{d,f}}, q_{\mathbf{v_\infty}}^{\pm \underline\lambda_{1,\infty}},\ldots, q_{\mathbf{v_\infty}}^{\pm \underline\lambda_{d,\infty}}\right]^{\underline{W}}.$$
Then $T$ annihilates $I_{P_M}^G(\sigma_\lambda)$ for all $\lambda\in \fa_M^*$, hence annihilates $L^2(G(F) \bs G(\BA)/K)_\fX$, but preserves $\pi$.
As there are only finitely many
$\fX = [(M,\sigma)]$ with $M \not= G$, a finite product of such $T$'s kills the orthogonal
space of the cuspidal spectrum $L_0^2(G(F) \bs G(\BA)/K)_\omega$ in $L^2(G(F) \bs G(\BA)/K)_\omega$, but does not kill $\pi$.

\paragraph{\bf Step 2: Isolating $\pi$} 
Recall that (Corollary \ref{cor of Harder}) there are only finitely many (equivalence classes of) cuspidal representations in the cuspidal spectrum
$L^2_0(G(F) \bs G(\BA)/K)$. 
Denote by $\pi_1,\ldots,\pi_n$ the cuspidal representations which
are not nearly equivalent to $\pi$. 
In particular, for $\pi_1$, there is a place $v_1$ of $F$ outside the union of $S$ and $\cup_{[(M,\sigma)]} S_{\sigma}$, such that $\pi_{1,v_1} \not\cong \pi_{v_1}$. 
Here $S_\sigma$ is a finite set of places such that $\mu_\sigma\in \CH_{S_{\sigma}}$, and $[(M, \sigma)]$ runs over all equivalence classes of cuspidal data.
It follows that we may find $T_{v_1} \in \CH_{v_1}$ such that
\[T_{v_1}(\chi_{\pi_{1,v_1}}) \not= T_{v_1}(\chi_{\pi_{v_1}}).\]
In particular, the Hecke element
\[T_{v_1} - T_{v_1}(\chi_{\pi_{1,v_1}}) \in \CH_{v_1}\]
kills $\pi_1$, but does not kill $\pi$. 
Continue this procedure for $\pi_2, \ldots, \pi_n$, we can construct a Hecke algebra element $\mu_0$ 
which kills all the cuspidal representations not nearly equivalent to $\pi$ in the spectrum, but does not kill $\pi$. 
Consider the finite product
\[\mu' = \mu_0 \cdot \prod_{[(M,\sigma)]} \mu_\sigma   \in \CH^{(S)}\]
where each $\mu_\sigma$ is constructed in {\bf Step 1} to kill $L^2(G(F) \bs G(\BA)/K)_{[(M,\sigma)]}$. Then $\mu'$ satisfies the
first condition in Theorem \ref{iso-cusp} which acts on $\pi^K$ by a non-zero constant. Finally, 
\[\mu = \pi(\mu')^{-1} \mu'\]
is a Hecke algebra element required in Theorem \ref{iso-cusp}.

\end{document}